\documentclass[a4paper, 10pt]{amsart}
\usepackage{amsmath}
\usepackage{amssymb}
\usepackage{fullpage}

\def\diag{\operatorname{diag}}

\def\id{\operatorname{id}}

\def\tr{\operatorname{tr}}

\newtheorem{cor}{Corollary}[section]
\newtheorem{prop}[cor]{Proposition}
\newtheorem{lem}[cor]{Lemma}
\newtheorem{thm}[cor]{Theorem}

\theoremstyle{remark}
\newtheorem{rem}[cor]{Remark}

\title{Geometric properties of domains related to $\mu$-synthesis}
\author{Pawe\l{} Zapa\l owski}

\address{Faculty of Mathematics and Computer Science, Jagiellonian University, \L o\-ja\-sie\-wi\-cza 6, 30-348 Krak\'ow, Poland}
\email{Pawel.Zapalowski@im.uj.edu.pl}


\begin{document}

\begin{abstract}In the paper we study the geometric properties of a large family of domains, called the generalized tetrablocks, related to the $\mu$-synthesis, containing both the family of the symmetrized polydiscs and the family of the $\mu_{1,n}$-quotients $\mathbb E_n$, $n\geq2$, introduced recently by G.~Bharali. It is proved that the generalized tetrablock cannot be exhausted by domains biholomorphic to convex ones. Moreover, it is shown that the Carath\'eodory distance and the Lempert function are not equal on a large subfamily of the generalized tetrablocks, containing i.a.~$\mathbb E_n$, $n\geq4$. We also derive a number of geometric properties of the generalized tetrablocks as well as the $\mu_{1,n}$-quotients. As a by-product, we get that the pentablock, another domain related to the $\mu$-synthesis problem introduced recently by J.~Agler, Z.~A.~Lykova, and N.~J.~Young, cannot be exhausted by domains biholomorphic to convex ones.
\end{abstract}

\maketitle

\section{Introduction}

A consequence of the celebrated Lempert theorem (cf.~\cite{l1981}) is the fact that if a domain $D$ can be exhausted by domains biholomorphic to convex ones, then the Carath\'eodory distance and the Lempert function coincide on $D$.

For more than 20 years it was an open conjecture that any bounded pseudoconvex domain $D$ with equality of the Carath\'eodory distance and the Lempert function can be exhausted by domains biholomorphic to convex ones.

Ten years ago A.~Agler and N.~J.~Young introduced domain $\mathbb G_2$, arising from the $\mu$-synthesis, called symmetrized bidisc (cf.~\cite{ay2004}). In 2007 A.~A.~Abouhajar, M.~C.~White, N.~J.~Young introduced another domain related to $\mu$-synthesis problem, called tetrablock and denoted by $\mathbb E$ (cf.~\cite{awy2007}). Both domains are bounded, hyperconvex (cf.~Section~\ref{sect:gt} for the definition of the hyperconvexity), and they cannot be exhausted by domains biholomorphic to convex ones. Nevertheless, the Lempert function and the Carath\'eodory distance coincide on them (see \cite{ay2004}, \cite{c2004a}, \cite{e2004}, \cite{ekz2013}, \cite{npz2008}). Further properties of these domains may be found in \cite{k2009}, \cite{pz2012} and \cite{z2013}.

$\mathbb G_2$ and $\mathbb E$ are---so far---the only counterexamples to the conjecture stated above.

A natural generalization of the symmetrized bidisc to higher dimensions is the symmetrized polydisc (cf.~\cite{c2005}). It turned out that in the family of the symmetrized polydiscs the symmetrized bidisc is the only counterexample for the converse to the Lempert theorem (see \cite{n2006}, \cite{nptz2008}, \cite{npz2007}). Further properties of the symmetrized polydisc may be found in \cite{ez2005}.

Recently G.~Bharali introduced another domain closely associated with an aspect of $\mu$-synthesis, denoted by $\mathbb E_n$ and called $\mu_{1,n}$-quotient, $n\geq2$ (cf.~\cite{b2014}). It is a natural generalization of the tetrablock, since $\mathbb E_2=\mathbb E$.

This article is devoted to studying the complex geometry of bounded domains related to the $\mu$-synthesis, which form a large family, containing both the family of the symmetrized polydiscs and the family of the $\mu_{1,n}$-quotients. The domains considered in the paper are generated by the space $E$ of the scalar block diagonal matrices (see the formula (\ref{eq:e}) below). We shall call them the generalized tetrablocks and denote by $\mathbb E_E$. In the engineering literature (e.g.~\cite{dp1993}) the space $E$ of matrices is usually taken to be given by a block diagonal structure, which partially justifies our choice. Let us mention here that such a choice of the space $E$ implies the logarithmic plurisubharmonicity of the structured singular value $\mu_E$ (cf.~Proposition~\ref{prop:psc}). The relation of the generalized tetrablocks to the $\mu$-synthesis problem will be explained in Section~\ref{sect:gt}.

Our first aim is to show that most of the generalized tetrablocks are not the counterexamples for the converse to the Lempert theorem. To be more precise, we show that the Carath\'eodory distance and the Lempert function are not equal on a large subfamily---denote it for a moment by $\mathcal E$---of the generalized tetrablocks (cf.~Proposition~\ref{prop:ecl}). We also show that none of the generalized tetrablock can be exhausted by domains biholomorphic to convex ones (cf.~Theorem~\ref{thm:e1}).

We also prove that any generalized tetrablock from the family $\mathcal E$ is neither $\mathbb C$-convex nor starlike about the origin, and that there is another subfamily of the generalized tetrablocks, containing i.a.~the $\mu_{1,n}$-quotients, such that each member of this subfamily is linearly convex, and hence pseudoconvex (cf.~Proposition~\ref{prop:lc}), hyperconvex and polynomially convex (cf.~Proposition~\ref{prop:pc}).

As an application, we get that in the family of the $\mu_{1,n}$-quotients, bounded hyperconvex domains, there are at most two counterexamples to the converse of the Lempert theorem. More precisely, the Carath\'eodory distance and the Lempert function are not equal on $\mathbb E_n$, $n\geq4$. Moreover, none of $\mathbb E_n$ can be exhausted by domains biholomorphic to convex ones (cf.~Theorem~\ref{thm:ee2en}, which collects also further properties of the $\mu_{1,n}$-quotients). All this properties make the family of the $\mu_{1,n}$-quotients very similar the family of the symmetrized polydiscs.

As a by-product of our considerations we get that the pentablock, another domain related to $\mu$-synthesis introduced recently by J.~Agler, Z.~A.~Lykova, and N.~J.~Young in \cite{aly2014}---although it is not generated by the space of the scalar block diagonal matrices---is hyperconvex and yet cannot be exhausted by domains biholomorphic to convex ones (cf.~Theorem~\ref{thm:3} and Proposition~\ref{prop:ph}).

Almost all results mentioned above are---more or less---easy consequence of the following, simple but powerful, fact saying that the generalized tetrablock $\mathbb E_{E'}$ generated by any subspace $E'$ of the vector space $E$ is an analytic retract of $\mathbb E_E$ (cf.~Theorem~\ref{thm:ret}). Another important tool we exploit in the paper are Propositions~\ref{prop:x} and \ref{prop:cl}, which originate in A.~Edigarian's paper \cite{e2013}. Since both propositions may be formulated in terms of arbitrary retracts, we put them into separate section.

The paper is organized as follows. In Section~\ref{sect:ar} we formulate two properties of general analytic retracts, we shall use in the sequel. In Section~\ref{sect:gt} we define the family of the generalized tetrablocks, show their relation to the $\mu$-synthesis problem, and give its geometric properties. In Section~\ref{sect:mq} we gather all results concerning the $\mu_{1,n}$-quotients, whereas the last section is devoted to the pentablock.

Here is some notation we shall use throughout the paper. By $\mathbb D$ we denote the open unit disc in the complex plane. Let $c_D$, $k_D$, and $l_D$ denote, respectively, the Carath\'eodory pseudodistance, the Ko\-ba\-ya\-shi pseudodistance, and the Lempert function of a domain $D\subset\mathbb C^n$ (for the definition and main properties of $c_D$, $k_D$, and $l_D$ the Reader may consult \cite{jp2013}). For $z=(z_1,\dots,z_n)\in\mathbb C^n$, $\lambda\in\mathbb C$ and $\alpha=(\alpha_1,\dots,\alpha_n)\in\mathbb Z_+^n$ we use the standard notation
$$
\lambda z:=(\lambda z_1,\dots,\lambda z_n),\quad z^{\alpha}:=z_1^{\alpha_1}\dots z_n^{\alpha_n}.
$$
Moreover, for $m=(m_1,\dots,m_n)\in\mathbb N^n$ and $\lambda\in\mathbb C$ denote the action on $\mathbb C^n$
$$
m_{\lambda}.z:=(\lambda^{m_1}z_1,\dots,\lambda^{m_n}z_n),\quad z=(z_1,\dots,z_n)\in\mathbb C^n.
$$
In the paper we will use the notion of quasibalanced domains. Recall that a domain $D\subset\mathbb C^n$ is called \emph{$m$-balanced}, where $m=(m_1,\dots,m_n)\in\mathbb N^n$, if $m_{\lambda}.z\in D$ whenever $z\in D$ and $\lambda\in\overline{\mathbb D}$. A $(1,\dots,1)$-balanced domain is called \emph{balanced}. A domain is called \emph{quasibalanced}, if it is $m$-balanced for some $m$.

\section{Analytic retracts}\label{sect:ar}

A domain $G$ is said to be an \emph{analytic retract} of a domain $D$ if there exist analytic maps $\theta:G\longrightarrow D$, $\iota:D\longrightarrow G$ such that $\iota\circ\theta=\id_G$.

For a domain $G$ by $\mathcal S(G)$ we denote the set of all holomorphic mappings $F:G\times G\longrightarrow G$ such that $F(z,z)=z$, $F(z,w)=F(w,z)$ for any $z,w\in G$.

Moreover, $G$ is called \emph{taut} if for any sequence $(f_j)_{j\in\mathbb N}$ of holomorphic mappings $f_j:\mathbb D\longrightarrow G$ there exists a subsequence $(f_{j_{\nu}})_{\nu\in\mathbb N}$ convergent uniformly on compact sets to a holomorphic mapping $f:\mathbb D\longrightarrow G$ or there exists a subsequence $(f_{j_{\nu}})_{\nu\in\mathbb N}$ that diverges uniformly on compact sets.

We shall make use of the following simple observation, which originates in A.~Edigarian's paper \cite{e2013} and is interesting in its own right.

\begin{prop}\label{prop:x}Let $G$ be an analytic retract of $D$ such that $\mathcal S(G)=\varnothing$. Then $D$ is not biholomorphic to a convex domain. If, additionally, $G$ is taut, then $D$ cannot be exhausted by domains biholomorphic to convex ones.
\end{prop}

\begin{proof}Suppose $\Omega$ is a convex domain and $f:D\longrightarrow\Omega$ is biholomorphic. By assumption, there are holomorphic mappings $\theta:G\longrightarrow D$, $\iota:D\longrightarrow G$ with $\iota\circ\theta=\id_G$. Define
$$
F(z,w):=\iota\circ f^{-1}\left(\frac{f\circ\theta(z)+f\circ\theta(w)}{2}\right),\quad z,w\in G.
$$
Observe that $F\in\mathcal S(G)$---a contradiction.

Now assume $G$ is taut. Suppose $D_1\subset D_2\subset\dots$, $\bigcup_{j\geq1}D_j=D$, $\Omega_j$ is a convex domain and $f_j:D_j\longrightarrow\Omega_j$ is biholomorphic, $j\geq1$. Define $G_j:=\theta^{-1}(D_j)$ and
$$
F_j(z,w):=\iota\circ f_j^{-1}\left(\frac{f_j\circ\theta(z)+f_j\circ\theta(w)}{2}\right),\quad z,w\in G_j,\ j\geq1.
$$
Observe that $F_j:G_j\times G_j\longrightarrow G$ with $F_j(z,w)=F_j(w,z)$, $F_j(z,z)=z$, $z,w\in G_j$, $j\geq1$. It follows easily from Montel's argument that there exists a holomorphic mapping $F:G\times G\longrightarrow\overline{G}$ such that $F(z,w)=F(w,z)$, $F(z,z)=z$, $z,w\in G$. Tautness of $G$ implies that either $F(G\times G)\subset G$ or $F(G\times G)\subset\partial G$. Since $F(z,z)=z\in G$, we conclude that the first case holds, i.e.~$F\in\mathcal S(G)$---a contradiction.
\end{proof}

Using holomorphic contractibility of the families of the Kobayashi pseudodistances and the Lempert functions we are able to prove

\begin{prop}\label{prop:cl}Let $G$ be an analytic retract of $D$ such that $l_G$ is not a distance. Then $l_D$ is not a distance. In particular, $c_D\not\equiv l_D$ and $D$ cannot be exhausted by domains biholomorphic to convex ones.
\end{prop}

\begin{proof}Suppose $l_D$ is a distance, i.e.~$k_D\equiv l_D$. By assumption, there are holomorphic mappings $\theta:G\longrightarrow D$, $\iota:D\longrightarrow G$ with $\iota\circ\theta=\id_G$. Then the holomorphic contractibility of the relevant families implies
\begin{align*}
k_G(x_1,x_2)&\geq k_D(\theta(x_1),\theta(x_2))=l_D(\theta(x_1),\theta(x_2))\\
{}&\geq l_G(\iota\circ\theta(x_1),\iota\circ\theta(x_2))=l_G(x_1,x_2),\quad x_1,x_2\in G,
\end{align*}
i.e.~$k_G\equiv l_G$---a contradiction.
\end{proof}

\section{The generalized tetrablock}\label{sect:gt}

Consider positive integers $n\geq2$, $s\leq n$, and $r_1,\dots,r_s$ with $\sum_{j=1}^sr_j=n$. In the set $A(r_1,\dots,r_s):=\{0,\dots,r_1\}\times\dots\times\{0,\dots,r_s\}\setminus\{(0,\dots,0)\}$ we introduce the following order. Given two different $\alpha=(\alpha_1,\dots,\alpha_s),\beta=(\beta_1,\dots,\beta_s)\in A(r_1,\dots,r_s)$ we write
\begin{equation}\label{eq:order}
\alpha<\beta\quad\textnormal{iff}\quad\alpha_{j_0}<\beta_{j_0},\textnormal{ where }j_0:=\max\{j:\alpha_j\neq\beta_j\}.
\end{equation}
Therefore we may write $A(r_1,\dots,r_s)=\{\alpha^1,\dots,\alpha^N\}$, where $\alpha^1<\dots<\alpha^N$ and $N:=\prod_{j=1}^s(r_j+1)-1$.

Finally, for $x=(x_1,\dots,x_N)\in\mathbb C^N$ and $z=(z_1,\dots,z_s)\in\mathbb C^s$ put
\begin{equation}\label{eq:R}
R_x(z):=1+\sum_{j=1}^N(-1)^{|\alpha^j|}x_jz^{\alpha^j},\quad
\end{equation}
and define
$$
\mathbb E_{n;s;r_1,\dots,r_s}:=\left\{x\in\mathbb C^N:\forall_{z\in\overline{\mathbb D}^s}\ R_x(z)\neq0\right\}.
$$
The set $\mathbb E_{n;s;r_1,\dots,r_s}$ we shall call the \emph{generalized tetrablock}.

\begin{rem}Note that $\mathbb E_{2;1;2}=\mathbb G_2$, $\mathbb E_{n;1;n}=\mathbb G_n$, $\mathbb E_{2;2;1,1}=\mathbb E$, and $\mathbb E_{n;2;n-1,1}=\mathbb E_n$.
\end{rem}

\subsection{Relation to the $\mu$-synthesis problem}

One of the central notions in the theory of robust control is the structured singular value, a matrix function denoted by $\mu$ and defined on $\mathbb C^{m\times n}$. In the definition of $\mu$ there is an underlying structure identified with linear subspace $E$ of $\mathbb C^{n\times m}$.

Let $E$ be a linear subspace of $\mathbb C^{n\times m}$. The \emph{structured singular value $\mu_E$ relative to $E$} is a function $\mu_E:\mathbb C^{m\times n}\longrightarrow\mathbb R_+$ given by
$$
\mu_E(A):=\frac{1}{\inf\{\|X\|:X\in E,\ \det(\mathbb I_n-AX)=0\}},\quad A\in\mathbb C^{m\times n},
$$
with the understanding that $\mu_E(A)=0$ if $\mathbb I_n-AX$ is always nonsingular. Here $\|\cdot\|$ denotes the operator norm. Recall that
\begin{itemize}
\item$\mu_E$ is upper semicontinuous,
\item$\mu_E(\lambda A)=|\lambda|\mu_E(A)$ for any $\lambda\in\mathbb C$, $A\in\mathbb C^{m\times n}$.
\end{itemize}
In particular,
\begin{equation*}
\Omega_{\mu_E}:=\{A\in\mathbb C^{n\times n}:\mu_E(A)<1\}
\end{equation*}
is a balanced domain and $\mu_E$ is its Minkowski functional (cf.~\cite{jp2013}, Remark~2.2.1).

The space $E$ is usually taken to be given by a block diagonal structure (cf.~\cite{dp1993} for basic properties of $\mu_E$ is this case). In this paper we consider only repeated scalar blocks. To be more precise, for a given positive integers $n\geq2$, $s\leq n$, and $r_1,\dots,r_s$ with $\sum_{j=1}^sr_j=n$, consider the vector subspace $E\subset\mathbb C^{n\times n}$ consisting of the following scalar block diagonal matrices
\begin{equation}\label{eq:e}
E=E(n;s;r_1,\dots,r_s):=\{\diag[z_1\mathbb I_{r_1},\dots z_s\mathbb I_{r_s}]\in\mathbb C^{n\times n}:z_1,\dots,z_s\in\mathbb C\}.
\end{equation}
Throughout the paper $E$ shall always denote the above subspace unless stated otherwise. For such a space $E$,
\begin{itemize}
\item$\rho=\mu_{E(n;1;n)}\leq\mu_E\leq\mu_{\mathbb C^{n\times n}}=\|\cdot\|$, where $\rho$ is the spectral radius,
\item$\mathbb B_{n\times n}\subset\Omega_{\mu_E}\subset\Omega_n$, where $\mathbb B_{n\times n}:=\{X\in\mathbb C^{n\times n}:\|X\|<1\}$ is the \emph{unit ball} and $\Omega_n:=\{X\in\mathbb C^{n\times n}:\rho(X)<1\}$ is the \emph{spectral ball},
\item$\mu_E$ is continuous,
\item$\mu_E(A)=\max_{X\in\overline{\mathbb B}_{n\times n}\cap E}\rho(XA)$ for any $A\in E$.
\end{itemize}

\begin{prop}\label{prop:psc}$\log\mu_E$ is continuous plurisubharmonic and $\Omega_{\mu_E}$ is pseudoconvex.
\end{prop}

\begin{proof}Recall that the spectral radius $\rho$ is plurisubharmonic  function (cf.~\cite{v1968}). Then the properties above imply that $\mu_E$ is plurisubharmonic. Now, as $\mu_E$ is plurisubharmonic Minkowski functional of the balanced domain $\Omega_{\mu_E}$, we conclude that $\Omega_{\mu_E}$ is pseudoconvex and $\log\mu_E$ is plurisubharmonic (cf.~\cite{jp2000}, Proposition 2.2.22).
\end{proof}

In the theory of robust control, the $\mu$-synthesis problem---an interpolation problem for analytic matrix functions, a generalization of the classical problems of Nevanlinna-Pick and Carath\'eodory-Fej\'er---is to construct an analytic matrix function $F:\mathbb D\longrightarrow\overline{\Omega}_{\mu_E}$ satisfying a finite number of interpolation conditions.

There is a natural relation between $\mathbb E_{n;s;r_1,\dots,r_s}$ and the domain $\Omega_{\mu_E}$.

For $j\leq n$ let $\mathcal J^j:=\{(i_1,\dots,i_j)\in\mathbb N^j:1\leq i_1<\dots<i_j\leq n\}$. Moreover, for $\alpha\in A(r_1,\dots,r_s)$ define
\begin{multline*}
\mathcal J^{|\alpha|}_{\alpha}:=\{(i_1,\dots,i_{|\alpha|})\in\mathcal J^{|\alpha|}:r_1+\dots+r_{j-1}+1\leq i_{\alpha_1+\dots+\alpha_{j-1}+1}<\\\dots<i_{\alpha_1+\dots+\alpha_j}\leq r_1+\dots+r_j,\ j=1,\dots,s\},
\end{multline*}
(recall here that $|\alpha|\leq n$). It is elementary to see that
$$
\bigcup_{j=1}^n\mathcal J^j=\bigcup_{j=1}^N\mathcal J^{|\alpha^j|}_{\alpha^j}\quad\textnormal{and}\quad\mathcal J^{|\alpha|}_{\alpha}\cap\mathcal J^{|\beta|}_{\beta}=\varnothing\textnormal{ whenever } \alpha\neq\beta.
$$
Finally, for $I\in\mathcal J^j$ and $A\in\mathbb C^{n\times n}$ let $A_I$ denotes the $j\times j$ submatrix of $A$ whose rows and columns are indexed by $I$.

Define a polynomial mapping $\pi_E:\mathbb C^{n\times n}\longrightarrow\mathbb C^N$ given by
\begin{equation*}
\pi_E(A):=\left(\sum_{I\in\mathcal J^{|\alpha^1|}_{\alpha^1}}\det A_I,\dots,\sum_{I\in\mathcal J^{|\alpha^N|}_{\alpha^N}}\det A_I\right).
\end{equation*}

\begin{prop}\label{prop:omega}$\pi_E(\Omega_{\mu_E})\subset\mathbb E_{n;s;r_1,\dots,r_s}$.
\end{prop}

In view of the above proposition, to shorten the notation, we shall write $\mathbb E_E:=\mathbb E_{n;s;r_1,\dots,r_s}$. In the proof we shall use the following

\begin{lem}[cf.~\cite{b2014}]\label{lem:2.1}If $A\in\mathbb C^{n\times n}$ then
$$
\det(\mathbb I_n-A\diag[z_1,\dots,z_n])=1+\sum_{j=1}^n(-1)^j\sum_{I\in\mathcal J^j}\det(A_I)z_I,
$$
where $z=(z_1,\dots,z_n)\in\mathbb C^n$ and $z_I:=z_{i_1}\dots z_{i_j}$ for $I=(i_1,\dots,i_j)$.
\end{lem}

\begin{proof}[Proof of Proposition~\ref{prop:omega}]Let $r>0$ and $A\in\mathbb C^{n\times n}$. Observe that $\mu_E(A)\leq1/r$ iff $\|X\|\geq r$ for any $X\in E$ with $\det(\mathbb I_n-AX)=0$.

For $z=(z_1,\dots,z_s)\in\mathbb C^s$ define $\underline{z}=(\underbrace{z_1,\dots,z_1}_{r_1\times},\dots,\underbrace{z_s,\dots,z_s}_{r_s\times})\in\mathbb C^n$. Note that
\begin{equation}\label{eq:zi}
\underline{z}_I=z^{\alpha},\quad I\in\mathcal J^{|\alpha|}_{\alpha},\ z\in\mathbb C^s,\ \alpha\in A(r_1,\dots,r_s).
\end{equation}

For any $X\in E$ there is $z=(z_1,\dots,z_s)\in\mathbb C^s$ such that
\begin{equation}\label{eq:x}
X=\diag[\underbrace{z_1,\dots,z_1}_{r_1\times},\dots,\underbrace{z_s,\dots,z_s}_{r_s\times}]\in\mathbb C^{n\times n}.
\end{equation}
Lemma~\ref{lem:2.1} together with (\ref{eq:zi}) implies that for $X$ given by (\ref{eq:x}) we have
\begin{equation}\label{eq:det}
\det(\mathbb I_n-AX)=1+\sum_{j=1}^N(-1)^{|\alpha^j|}\left(\sum_{I\in\mathcal J^{|\alpha^j|}_{\alpha^j}}\det A_I\right)z^{\alpha^j},
\end{equation}
Hence, by (\ref{eq:det}), $\mu_E(A)\leq1/r$ iff the zero variety of the polynomial (\ref{eq:det}) in $z_1,\dots,z_s$ does not meet the open polydisc $(r\mathbb D)^s$.

Suppose that $\mu_E(A)<1$ and $x=\pi_E(A)$. For some $r>1$ we have $\mu_E(A)\leq 1/r$, and so the zero variety of (\ref{eq:det}) is disjoint from $(r\mathbb D)^s$ and, consequently, from $\overline{\mathbb D}^s$. Thus $x\in \mathbb E_E$.
\end{proof}

\begin{rem}(a) If $n=2$, $s=1$, i.e.~$E:=\{z\mathbb I_2\in\mathbb C^{2\times 2}:z\in\mathbb C\}$, then $N=2$ and
$$
\pi_E(A)=(\tr A,\det A),\quad A\in\mathbb C^{2\times 2}.
$$

(b) If $n=s=2$, i.e.~$E:=\{\diag[z_1,z_2]\in\mathbb C^{2\times 2}:z_1,z_2\in\mathbb C\}$, then $N=3$ and
$$
\pi_E(A)=(a_{1,1},a_{2,2},\det A),\quad A=[a_{j,k}]_{j,k=1}^2in\mathbb C^{2\times 2}.
$$

(c) More general, if $s=2$, $r_1=n-1$, $r_2=1$, i.e.~$E:=\{\diag[z_1\mathbb I_{n-1},z_2]\in\mathbb C^{n\times n}:z_1,z_2\in\mathbb C\}$, then $N=2n-1$ and
\begin{multline*}
\pi_E(A)=\left(\sum_{I\in\mathcal J^1:i_1\geq2}\det A_I,\dots,\sum_{I\in\mathcal J^{n-1}:i_1\geq2}\det A_I,\right.\\\left.\sum_{I\in\mathcal J^1:i_1=1}\det A_I,\dots,\sum_{I\in\mathcal J^n:i_1=1}\det A_I\right),\quad A\in\mathbb C^{n\times n}.
\end{multline*}

(d) Recall that $\pi_E(\Omega_{\mu_E})=\mathbb E_E$ for $\mathbb E_E\in\{\mathbb G_n,\mathbb E_n\}$, $n\geq2$. It is an open question whether this equality holds for general $\mathbb E_E$.

(e) About 15 years ago J.~Agler and N.~J.~Young in \cite{ay1999} devised a new approach to the Nevanlinna--Pick interpolation problem for $\Omega_{\mu_E}$. They reduced the given analytic interpolation problem for $\Omega_{\mu_E}$-valued functions with to one for $\mathbb G_2$-valued functions (if $n=2$ and $s=1$) or $\mathbb E$-valued functions (if $n=s=2$). Recently, G.~Bharali applied this reduction strategy in the case of $\mathbb E_n$-valued functions (if $s=2$, $r_1=n-1$, $r_2=1$). Previous attempts to find analysable instances of $\mu$-synthesis have led to the study of the symmetrized bidisc, the tetrablock and the $\mu_{1,n}$-quotients. First two of these domains have turned out to have interesting function-theoretic properties. The genesis of this paper was to examine to what extend properties of $\mathbb G_n$ and $\mathbb E$ are inherited by their natural generalizations such as $\mu_{1,n}$-quotients $\mathbb E_n$ or the co-called generalized tetrablocks $\mathbb E_E$.

(f) Observe that $n\leq N\leq2^n-1$. Moreover, if $s=1$ then $N=n$, whereas for $s=n$ we have $N=2^n-1$.

(g) Recall that one of two major effects of the idea introduced by J.~Agler and N.~J.~Young is the reduction in the dimensional complexity of the Nevanlinna--Pick interpolation problem for $\Omega_{\mu_E}$. (f) shows that this advantage disappears completely as the number of scalar blocks in $E$ increases. Moreover, the dimension may significantly increase when passing form $\Omega_{\mu_E}$ to $\mathbb E_E$ as $n^2\ll2^n-1$ for big $n$.
\end{rem}

\subsection{Geometry of the generalized tetrablock}

\begin{prop}\label{cor:balanced}$\mathbb E_E$ is bounded $(|\alpha^1|,\dots,|\alpha^N|)$-balanced domain.
\end{prop}

\begin{proof}First we show that $\mathbb E_E$ is $|\alpha|$-balanced, where $|\alpha|:=(|\alpha^1|,\dots,|\alpha^N|)$. Take $x=(x_1,\dots,x_N)\in\mathbb E_E$ and $\lambda\in\overline{\mathbb D}$. Our aim is to show $|\alpha|_{\lambda}.x\in\mathbb E_E$, i.e.
$$
R_{|\alpha|_{\lambda}.x}(z)\neq0,\quad z=(z_1,\dots,z_s)\in\overline{\mathbb D}^s.
$$
But it is an immediate consequence of the following equality
$$
R_{|\alpha|_{\lambda}.x}(z)= 1+\sum_{j=1}^N(-1)^{|\alpha^j|}\lambda^{|\alpha^j|}x_jz^{\alpha^j}=1+\sum_{j=1}^N(-1)^{|\alpha^j|}x_j(\lambda z)^{\alpha^j}=R_x(\lambda z).
$$
It remains to observe that $R_x(\lambda z)\neq0$, since $\lambda z\in\overline{\mathbb D}^s$ for any $z\in\overline{\mathbb D}^s$. Thus $\mathbb E_E$ is $(|\alpha^1|,\dots,|\alpha^N|)$-balanced set.

Since $\mathbb E_E$ is open by definition, we conclude that $\mathbb E_E$ is $(|\alpha^1|,\dots,|\alpha^N|)$-balanced domain.

To see $\mathbb E_E$ is bounded we proceed as follows. Take $m=(m_1,\dots,m_s)\in\mathbb N^s$ such that
$$
z_0^{\langle m,\alpha^j\rangle}\neq z_0^{\langle m,\alpha^k\rangle},\quad j,k=1,\dots,N,\ j\neq k,\ z_0\in\overline{\mathbb D}\setminus\{0\},
$$
and
$$
(-1)^{|\alpha^j|}=(-1)^{\langle m,\alpha^j\rangle},\quad j=1,\dots,N.
$$
Put $z(z_0,m):=(z_0^{m_1},\dots,z_0^{m_s})$, $z_0\in\overline{\mathbb D}$. Take $(x_1,\dots,x_N)\in\mathbb E_E$. Since
$$
R_x(z(z_0,m))=1+\sum_{j=1}^N(-1)^{\langle m,\alpha^j\rangle}x_jz_0^{\langle m,\alpha^j\rangle}\neq0,\quad z_0\in\overline{\mathbb D},
$$
we conclude that $\widetilde x:=(\widetilde x_1,\dots,\widetilde x_M)\in\mathbb G_M$ with $M:=\max\{\langle m,\alpha^j\rangle:j=1,\dots,N\}$ and
$$
\widetilde x_k:=\begin{cases}x_k,\quad&\textnormal{if there is }j\textnormal{ such that }\langle m,\alpha^j\rangle=k\\0,\quad&\textnormal{otherwise}\end{cases},\quad k=1,\dots,M.
$$
The boundedness of the symmetrized polydisc $\mathbb G_M$ finishes the proof.
\end{proof}

Let
\begin{equation}\label{eq:e'}
E':=\{\diag[z_1\mathbb I_{r_1},\dots z_{s'}\mathbb I_{r_{s'}}]\in\mathbb C^{n'\times
n'}:z_1,\dots,z_{s'}\in\mathbb C\},
\end{equation}
for some $s'\leq s$ and $n':=\sum_{j=1}^{s'}r_j$. Let $N'$ be such that
$$
\alpha^j_{\nu}=0,\quad 1\leq j\leq N',\ s'<\nu\leq s,\qquad\textnormal{and}\qquad \alpha^{N'+1}_{s'+1}\neq0.
$$
Observe that $N'=\prod_{j=1}^{s'}(r_j+1)-1$. We define
$$
(\alpha^j)':=(\alpha^j_1,\dots,\alpha^j_{s'}),\quad j=1,\dots,N'.
$$
For $x\in\mathbb C^N$ write $x=(x',x'')\in\mathbb C^{N'}\times\mathbb C^{N''}$, where $N'':=N-N'$. For $z\in\mathbb C^s$ write $z=(z',z'')\in\mathbb C^{s'}\times\mathbb C^{s-s'}$. Finally, for $x'=(x_1,\dots,x_{N'})$ define
\begin{equation}\label{eq:R'}
R'_{x'}(z'):=1+\sum_{j=1}^{N'}(-1)^{|(\alpha^j)'|}x_j(z')^{(\alpha^j)'},\quad z'\in\mathbb C^{s'}.
\end{equation}
Then
$$
\mathbb E_{E'}=\left\{x'\in\mathbb C^{N'}:\forall_{z'\in\overline{\mathbb D}^{s'}}\ R'_{x'}(z')\neq0\right\}.
$$
Throughout the paper $E'$ will always denote the "subspace" (\ref{eq:e'}) of the space $E$ given by (\ref{eq:e}) unless stated otherwise. We start with elementary but crucial

\begin{thm}\label{thm:ret}The mappings
\begin{equation}\label{eq:ti}
\mathbb E_{E'}\ni x'\overset{\theta}\longmapsto(x',0)\in\mathbb E_E,\quad \mathbb E_E\ni(x',x'')\overset{\iota}\longmapsto x'\in\mathbb E_{E'},
\end{equation}
are well defined. In particular, $\mathbb E_{E'}$ is an analytic retract of $\mathbb E_E$. Moreover, $\mathbb E_E$ is a Hartogs domain over $\mathbb E_{E'}$ with $N''$-dimensional $m$-balanced fibers, where
$$
m=\left(\underbrace{|\alpha^{N'+1}|,\dots,|\alpha^{N'+1}|}_{(N'+1)\times},\dots, \underbrace{|\alpha^{M(N'+1)}|,\dots,|\alpha^{M(N'+1)}|}_{(N'+1)\times}\right)\in\mathbb Z^{N''},
$$
$M=\prod_{j=s'+1}^s(r_j+1)-1$.
\end{thm}

\begin{proof}
Let $x'\in\mathbb E_{E'}$. Consider the point $(x',0)\in\mathbb C^N$. Then
$$
R_{(x',0)}(z)=1+\sum_{j=1}^{N'}(-1)^{|\alpha^j|}x_jz^{\alpha^j}=1+\sum_{j=1}^{N'}(-1)^{|(\alpha^j)'|}x_j(z')^{(\alpha^j)'}=R'_{x'}(z'),
$$
for any $z=(z',z'')\in\mathbb C^s$. Consequently, since $R_{x'}(z')\neq0$ for any $z'\in\overline{\mathbb D}^{s'}$ then also $R'_{(x',0)}(z)\neq0$ for any $z=(z',z'')\in\overline{\mathbb D}^s$, i.e.~$(x',0)\in\mathbb E_E$. Hence $\theta$ is well defined.

Now take $x\in\mathbb E_E$. Directly from the definition of $\mathbb E_E$ it follows that $R_x(z',0)\neq0$ for all $z'\in\overline{\mathbb D}^{s'}$. Note that
$$
R_x(z',0)=1+\sum_{j=1}^{N'}(-1)^{|\alpha^j|}x_jz^{\alpha^j}=1+\sum_{j=1}^{N'}(-1)^{|(\alpha^j)'|}x_j(z')^{(\alpha^j)'}=R'_{x'}(z'),
$$
whence $x'\in\mathbb E_{E'}$, i.e.~$\iota$ is well defined, too.

So far we know that
$$
\mathbb E_E\cap(\mathbb C^{N'}\times\{0\}^{N''})=\mathbb E_{E'}\times\{0\}^{N''}.
$$
To see that $\mathbb E_E$ is a Hartogs domain over $\mathbb E_{E'}$ with $N''$-dimensional $m$-balanced fibers we proceed as follows. For $x'\in\mathbb E_{E'}$ define the fiber $D_{x'}:=\{x''\in\mathbb C^{N''}:(x',x'')\in\mathbb E_E\}$. It remains to see that $D_{x'}$ is $m$-balanced, $x'\in\mathbb E_{E'}$. Recall that
$$
N''=N-N'=\prod_{j=1}^{s'}(r_1+1)\left(\prod_{j=s'+1}^s(r_1+1)-1\right)=(N'+1)M.
$$
Let
$$
\beta^j:=(\alpha^{j(N'+1)}_{s'+1},\dots,\alpha^{j(N'+1)}_s),\quad j=1,\dots,M,
$$
and observe that $|\beta^j|=|\alpha^{j(N'+1)}|$, $j=1,\dots,M$.

Fix $x'\in\mathbb E_{E'}$ and $x''=(x_1,\dots,x_{N''})\in D_{x'}$. We aim at showing that
$$
|\beta|_{\lambda}.x''\in D_{x'},\quad \lambda\in\overline{\mathbb D},
$$
where $|\beta|:= (\underbrace{|\beta^1|,\dots,|\beta^1|}_{(N'+1)\times},\dots,\underbrace{|\beta^M|,\dots,|\beta^M|}_{(N'+1)\times})$. In other words, we want to show that
\begin{equation}\label{eq:Rbeta}
R_{(x',|\beta|_{\lambda}.x'')}(z)\neq0,\quad z\in\overline{\mathbb D}^s,\ \lambda\in\overline{\mathbb D}.
\end{equation}
But it is an immediate consequence of the following equality
\begin{equation}\label{eq:Rsplit}
R_x(z)=R'_{x'}(z')+\sum_{k=1}^M(-1)^{|\beta^k|}(z'')^{\beta^k}\sum_{j=0}^{N'}(-1)^{|(\alpha^j)'|}x_{k(N'+1)+j}(z')^{(\alpha^j)'}.
\end{equation}
Indeed, using (\ref{eq:Rsplit}) we get
\begin{align*}
R_{(x',|\beta|_{\lambda}.x'')}(z)& =R'_{x'}(z')+\sum_{k=1}^M(-1)^{|\beta^k|}(z'')^{\beta^k}\sum_{j=0}^{N'}(-1)^{|(\alpha^j)'|}\lambda^{|\beta^k|}x_{k(N'+1)+j}(z')^{(\alpha^j)'}\\ {}&=R'_{x'}(z')+\sum_{k=1}^M(-1)^{|\beta^k|}(\lambda z'')^{\beta^k}\sum_{j=0}^{N'}(-1)^{|(\alpha^j)'|}x_{k(N'+1)+j}(z')^{(\alpha^j)'}\\{}&=R_{(x',x'')}(z',\lambda z'').
\end{align*}
Hence and from the fact that $(z',\lambda z'')\in\overline{\mathbb D}^s$ for any $\lambda\in\overline{\mathbb D}$ and $(z',z'')\in\overline{\mathbb D}^s$ we get (\ref{eq:Rbeta}).
\end{proof}

\begin{rem}Note that in the above theorem instead of first $s'$ blocks $r_1,\dots,r_{s'}$ that define the subspace $E'$ one may take arbitrary subset $\{r_{j_1},\dots,r_{j_{s'}}\}$ of $\{r_1,\dots,r_s\}$.
\end{rem}

\begin{cor}\label{cor:retgn}$\mathbb G_{\max\{2,r_1,\dots,r_s\}}$ is an analytic retract of $\mathbb E_E$.
\end{cor}

\begin{proof}If $s=1$ then $\mathbb E_E=\mathbb G_n$ and we are done. So assume that $s>1$. If there is $j$ with $r_j=\max\{2,r_1,\dots,r_s\}$, without loss of generality we may assume that $r_1=\max\{r_1,\dots,r_s\}$ and define $E':=\{z\mathbb I_{r_1}:z\in\mathbb C\}$. Then Theorem~\ref{thm:ret} implies that $\mathbb E_{E'}=\mathbb G_{r_1}$ is an analytic retract of $\mathbb E_E$.

Otherwise $s=n$ and $r_1=\dots=r_n=1$. Then we define $E':=\{\diag[z_1,z_2]:z_1,z_2\in\mathbb C\}$ and either $\mathbb E_E=\mathbb E_{E'}$ or Theorem~\ref{thm:ret} implies that $\mathbb E_{E'}=\mathbb E$ is an analytic retract of $\mathbb E_E$. Moreover, $\mathbb G_2$ is an analytic retract of $\mathbb E$. Indeed, to see this consider the analytic mappings
\begin{equation*}
\mathbb G_2\ni(s,p)\overset{\theta}\longmapsto\left(\frac{s}{2},\frac{s}{2},p\right)\in\mathbb E,\quad \mathbb E\ni(x_1,x_2,x_3)\overset{\iota}\longmapsto(x_1+x_2,x_3)\in\mathbb G_2,
\end{equation*}
whence $\iota\circ\theta=\id_{\mathbb G_2}$. Consequently, $\mathbb G_2$ is an analytic retract of $\mathbb E_E$.
\end{proof}

\begin{prop}\label{prop:ecl}Assume there is $j$ such that $r_j\geq3$. Then $l_{\mathbb E_E}$ is not a distance. In particular, $c_{\mathbb E_E}\not\equiv l_{\mathbb E_E}$.
\end{prop}

\begin{proof}If $s=1$ then $\mathbb E_E=\mathbb G_n$, $n\geq3$, and we are done. So assume that $s>1$. Without loss of generality we may assume that $r_1\geq3$. We apply Theorem~\ref{thm:ret} to $E'$ with $s'=1$. Then we use Proposition~\ref{prop:cl} with $G=\mathbb E_{E'}=\mathbb G_{r_1}$, $D=\mathbb E_E$ and the fact that $l_{\mathbb G_{r_1}}$ is not a distance (cf.~\cite{npz2007}).
\end{proof}

Moreover, in some cases we get more precise information.

\begin{prop}\label{prop:eck}If there is $j$ such that $r_j=3$ then $c_{\mathbb E_E}(0,\cdot)\not\equiv k_{\mathbb E_E}(0,\cdot)$. In particular, $c_{\mathbb E_E}(0,\cdot)\not\equiv l_{\mathbb E_E}(0,\cdot)$.
\end{prop}

\begin{proof}If $s=1$ then $\mathbb E_E=\mathbb G_3$ and we are done. So assume that $s>1$. Without loss of generality we may assume that $r_1=3$. We apply Theorem~\ref{thm:ret} to $E'$ with $s'=1$. Then we use Proposition~\ref{prop:cl} with $G=\mathbb E_{E'}=\mathbb G_3$, $D=\mathbb E_E$ and the fact that $c_{\mathbb G_3}(0,\cdot)\not\equiv k_{\mathbb G_3}(0,\cdot)$ (cf.~\cite{nptz2008}).
\end{proof}

Now we are in position to prove the following

\begin{thm}\label{thm:e1}$\mathbb E_E$ cannot be exhausted by domains biholomorphic to convex ones.
\end{thm}

\begin{proof}In view of Proposition~\ref{prop:ecl} it suffices to consider $\max\{r_1,\dots,r_s\}\leq2$. From Corollary~\ref{cor:retgn} it follows that $\mathbb G_2$ is an analytic retract of $\mathbb E_E$. Moreover, $\mathbb G_2$ is taut and $\mathcal S(\mathbb G_2)=\varnothing$ (cf.~\cite{e2013}, Corollary 3). It remains to apply Proposition~\ref{prop:x}.
\end{proof}

We conclude this subsection with some further basic geometric properties of the generalized tetrablocks $\mathbb E_E$.

\begin{cor}\label{cor:circ}$\mathbb E_E$ is not circled.
\end{cor}

\begin{proof}Corollary~\ref{cor:retgn} implies that, after the permutation of the variables if necessary, there is $n\geq2$ such that $(x,0)\in\mathbb E_E$ iff $x\in\mathbb G_n$. It remains to observe that $\mathbb G_n$ is not circled.
\end{proof}

Recall that a domain $\mathbb D\subset\mathbb C^n$ is called (cf.~\cite{h1994}, \cite{aps2004})
\begin{itemize}
\item$\mathbb C$-\emph{convex} if for any affine complex line $L$ such that $L\cap D\neq\varnothing$, the set $L\cap D$ is connected and simply connected;
\item\emph{linearly convex} if its complement is a union of affine complex hyperplanes.
\end{itemize}
Note that any $\mathbb C$-convex domain is linearly convex.

\begin{prop}\label{prop:ncc}If there is $j$ such that $r_j\geq3$ then $\mathbb E_E$ is neither $\mathbb C$-convex nor starlike about the origin.
\end{prop}

\begin{proof}Without loss of generality we may assume that $r_1\geq3$. Let $E'$ be given by (\ref{eq:e'}) with $s'=1$. It follows from Theorem~\ref{thm:ret} that
\begin{equation}\label{eq:section}
\mathbb E_E\cap(\mathbb C^{r_1}\times\{0\}^{N-r_1})=\mathbb G_{r_1}\times\{0\}^{N-r_1}.
\end{equation}
Since $\mathbb G_{r_1}$ is not $\mathbb C$-convex (cf.~\cite{npz2008}), there is an affine complex line $L'\subset\mathbb C^{r_1}$ such that $L'\cap\mathbb G_{r_1}\neq\varnothing$ and the set $L'\cap\mathbb G_{r_1}$ either is not connected or is not simply connected. Consequently, $L:=L'\times\{0\}^{N-r_1}\subset\mathbb C^N$ is an affine complex line such that $L\cap\mathbb E_E\neq\varnothing$ and the set $L\cap\mathbb E_E$ either is not connected or is not simply connected.

To see $\mathbb E_E$ is not starlike about the origin, use (\ref{eq:section}) and the fact that $\mathbb G_{r_1}$ is not starlike about the origin (cf.~\cite{npz2008}).
\end{proof}

In \cite{y2008} N.~J.~Young showed that $\mathbb E$ is not an analytic retract of the open unit ball of a $\mathcal J^*$-algebra of finite rank (see \cite{h1973} for a definition of a $J^*$-algebra). By careful analysis of Young's proof \L.~Kosi\'nski showed in \cite{k2014} that the same property holds for $\mathbb G_2$ (and hence for $\mathcal P$, since $\mathbb G_2$ is an analytic retract of $\mathcal P$). Consequently, we get

\begin{prop}\label{prop:ar}Assume that there is $j$ such that $r_j=2$ or there are $j,k$, $j\neq k$, such that $r_j=r_k=1$. Then $\mathbb E_E$ is not an analytic retract of the open unit ball of a $\mathcal J^*$-algebra of finite rank.
\end{prop}

\subsection{The case $r_2=\dots=r_s=1$}Assume $s\geq2$. Let $E'$ be defined as in (\ref{eq:e'}) with $s'=s-1$. We write the polynomial (\ref{eq:R}) defining $\mathbb E_E$ in the form
$$
R_x(z)=R'_{x'}(z')-z_sP_{x''}(z')
$$
and define the rational function
$$
\Psi_{z'}(x):=\frac{P_{x''}(z')}{R'_{x'}(z')},\quad x=(x',x'')\in\mathbb C^{N'}\times\mathbb C^{N''},\ z'\in\mathbb C^{s-1},\ R'_{x'}(z')\neq0.
$$

There is the following immediate characterization of such $\mathbb E_E$, analogous to the one for the $\mu_{1,n}$-quotient $\mathbb E_n$ (cf.~\cite{awy2007} and \cite{b2014}).

\begin{prop}\label{prop:genbha}Let $s\geq2$, $r_s=1$, $E'$ be as in (\ref{eq:e'}) with $s'=s-1$, and let $(x',x'')\in\mathbb C^{N'}\times\mathbb C^{N''}$. Then the following are equivalent
\begin{enumerate}
\item[(i)]$(x',x'')\in\mathbb E_E$;
\item[(ii)]$x'\in\mathbb E_{E'}$, the function $z'\mapsto\Psi_{z'}(x',x'')$ is holomorphic on $\mathbb D^{s-1}$, continuous on $\overline{\mathbb D}^{s-1}$, and
$$
\max_{z'\in\overline{\mathbb D}^{s-1}}|\Psi_{z'}(x',x'')|=\max_{z'\in\mathbb T^{s-1}}|\Psi_{z'}(x',x'')|<1.
$$
\end{enumerate}
\end{prop}

\begin{proof}Condition (i) is equivalent to
$$
R'_{x'}(z')\neq z_sP_{x''}(z'),\quad z'\in\overline{\mathbb D}^{s-1},\ z_s\in\overline{\mathbb D},
$$
 i.e.~$x'\in\mathbb E_{E'}$ and $1\neq z_s\Psi_{z'}(x',x'')$ for all $z'\in\overline{\mathbb D}^{s-1}$, which is equivalent to (ii).
\end{proof}

Let $\overline{\mathbb E}_E$ denote the closure of $\mathbb E_E$. Similarly we obtain

\begin{prop}\label{prop:clo}Let $s\geq2$, $r_2=\dots=r_s=1$, $E'$ be as in (\ref{eq:e'}) with $s'=s-1$, and let $x=(x',x'')\in\mathbb C^{N'}\times\mathbb C^{N''}$. Then the following are equivalent
\begin{enumerate}
\item[(i)]$R_x(z)\neq0$ for any $z\in\mathbb D^s$;
\item[(ii)]$x\in\overline{\mathbb E}_E$;
\item[(iii)]$x'\in\overline{\mathbb E}_{E'}$, the function $z'\mapsto\Psi_{z'}(x',x'')$ is holomorphic on $\mathbb D^{s-1}$, and
    $$
    \sup_{z'\in\mathbb D^{s-1}}|\Psi_{z'}(x',x'')|\leq1.
    $$
\end{enumerate}
\end{prop}

\begin{proof}(i)$\Rightarrow$(ii) Take $z\in\overline{\mathbb D}^s$. Since $rz\in\mathbb D^s$ for any $r\in(0,1)$, (i) implies $R_x(rz)\neq0$. Recall that $R_x(rz)=R_{|\alpha|_r.x}(z)$, whence $|\alpha|_r.x\in\mathbb E_E$ for $r\in(0,1)$, i.e.~$x\in\overline{E}_E$.

(ii)$\Rightarrow$(i) Using induction on $k$ we prove that for any $k=2,\dots,n$
\begin{equation}\label{eq:impl}
x\in\overline{E}_{E_k}\Rightarrow R_x(z)\neq0\textnormal{ for any }z\in\mathbb D^k,
\end{equation}
where
\begin{equation}\label{eq:ek}
E_k=\{\diag[z_1\mathbb I_{r_1},z_2,\dots,z_k]:z_1,\dots,z_k\in\mathbb C\},\quad k=2,\dots,s.
\end{equation}

Take $k=2$ and suppose $x\in\overline{E}_{E_2}$ but $0=R_x(z)=R'_{x'}(z_1)-z_2P_{x''}(z_1)$ for some $z_1,z_2\in\mathbb D$. Then $x'\in\overline{\mathbb G}_{r_1}$, i.e.~$R'_{x'}(z_1)\neq0$. Consequently, $1=z_2\Psi_{z_1}(x)$ and so $|\Psi_{z_1}(x)|>1$. However, $|\Psi_{z_1}(y)|<1$ for all $y\in\mathbb E_{E_2}$ and since $x$ is a limit point of such $y$ we have $|\Psi_{z_1}(x)|\leq1$, a contradiction.

Now fix $2\leq k<n$ and assume the implication (\ref{eq:impl}) holds for $k$. Suppose $x\in\overline{E}_{E_{k+1}}$ but $0=R_x(z)=R'_{x'}(z')-z_{k+1}P_{x''}(z')$ for some $z'\in\mathbb D^k$, $z_{k+1}\in\mathbb D$. Then $x'\in\overline{\mathbb E}_{E_k}$, i.e.~$R'_{x'}(z')\neq0$ by inductive assumption. Consequently, we may proceed as in the the case $k=2$.

The proof of (ii)$\Leftrightarrow$(iii) is much as for Proposition~\ref{prop:genbha}.
\end{proof}

Argument used in \cite{k2014} to prove linear convexity of the pentablock allows us to get

\begin{prop}\label{prop:lc}If $r_2=\dots=r_s=1$ then $\mathbb E_E$ is linearly convex and, consequently, pseudoconvex.
\end{prop}

\begin{proof}Without loss of generality we may assume that $s\geq2$. Using induction on $k$ we prove that $\mathbb E_{E_k}\subset\mathbb C^{N_k}$ given by (\ref{eq:ek}) is linearly convex, $k=2,\dots,s$.

First we show that $\mathbb E_{E_2}$ is linearly convex. Take $x_0\notin\mathbb E_{E_2}$. We are looking for a complex hyperplane $H\subset\mathbb C^{N_2}$ such that $x_0\in H$ and $H\cap\mathbb E_{E_2}=\varnothing$. It follows from Theorem~\ref{thm:ret} that if we write $x=(x',x'')\in\mathbb E_{E_2}\subset\mathbb C^{r_1}\times\mathbb C^{N_2-r_1}$ then $x'\in\mathbb G_{r_1}$.

Assume first that $x_0'\notin\mathbb G_{r_1}$. Since $\mathbb G_{r_1}$ is linearly convex (cf.~\cite{npz2008}), there is a complex hyperplane $H'\subset\mathbb C^{r_1}$ such that $x_0'\in H'$ and $H'\cap\mathbb G_{r_1}=\varnothing$. Then $H:=H'\times\mathbb C^{N_2-r_1}$ is a complex hyperplane we are looking for (use Theorem~\ref{thm:ret}).

Now consider the case $x_0'\in\mathbb G_{r_1}$. It follows from Proposition~\ref{prop:genbha} that $\Psi_{z_1}(x_0',x'')$ is well defined for any $z_1\in\overline{\mathbb D}$ and $x''\in\mathbb C^{N_2-r_1}$. Moreover, applying Proposition~\ref{prop:genbha} we get that there is a $z_1\in\overline{\mathbb D}$ and $\omega\in\mathbb C\setminus\mathbb D$ such that $\Psi_{z_1}(x_0',x_0'')=\omega$ and $\Psi_{z_1}(x',x'')\neq\omega$ whenever $(x',x'')\in\mathbb E_{E_2}$. Thus
$$
H:=\{(x',x'')\in\mathbb C^{r_1}\times\mathbb C^{N_2-r_1}:\Psi_{z_1}(x',x'')=\omega\}
$$
is a complex hyperplane satisfying desired properties.

Now fix $2\leq k<s$ and assume that $\mathbb E_{E_k}$ is linearly convex. In order to show that $\mathbb E_{E_{k+1}}$ is linearly convex, we take $x_0\notin\mathbb E_{E_{k+1}}$ and look for a complex hyperplane $H\subset\mathbb C^{N_{k+1}}$ such that $x_0\in H$ and $H\cap\mathbb E_{E_{k+1}}=\varnothing$. Since $\mathbb E_{E_k}$ is assumed to be linearly convex we may proceed as in first inductive step replacing $\mathbb G_{r_1}$ with $\mathbb E_{E_k}$.

Pseudoconvexity of $\mathbb E_E$ is a consequence of Proposition~2.1.8 from \cite{aps2004}.
\end{proof}

For a given $m=(m_1,\dots,m_n)\in\mathbb N^n$ and an $m$-balanced domain $D\subset\mathbb C^n$ define its $m$-\emph{Minkowski functional}
$$
\mathfrak h_D(x):=\inf\{\lambda>0:m_{\lambda}.x\in D\},\quad x\in\mathbb C^n.
$$
This function has similar properties as the standard Minkowski functional for balanced domains. Some of them may be found in \cite{jp2013}. In particular,
\begin{itemize}
\item$\mathfrak h_D(m_{\lambda}.x)=|\lambda|\mathfrak h_D(x)$, $x\in\mathbb C^n$, $\lambda\in\mathbb C$,
\item$D=\{x\in\mathbb C^n:\mathfrak h_D(x)<1\}$,
\item if $D$ is bounded and $\partial D=\{x\in\mathbb C^n:\mathfrak h_D(x)=1\}$ then $\mathfrak h_D$ is continuous.
\end{itemize}

A bounded domain is called \emph{hyperconvex} if there exists a continuous negative plurisubharmonic exhaustion function. In particular, any hyperconvex domain is taut.

\begin{prop}\label{prop:hc}If $r_2=\dots=r_s=1$ then $\mathfrak h_{\mathbb E_E}$ is continuous. In particular, $\mathbb E_E$ is hyperconvex.
\end{prop}

\begin{proof}To prove the continuity of $\mathfrak h_{\mathbb E_E}$ it suffices to show that
$$
\partial\mathbb E_E\subset\{x\in\mathbb C^N:\mathfrak h_{\mathbb E_E}(x)=1\}.
$$
Suppose there is $x\in\partial\mathbb E_E$ such that $\mathfrak h_{\mathbb E_E}(x)>1$. Take $0<\lambda<1$ such that $\mathfrak h_{\mathbb E_E}(|\alpha|_{\lambda}.x)=\lambda\mathfrak h_{\mathbb E_E}(x)>1$. In particular, $|\alpha|_{\lambda}.x\notin\mathbb E_E$. On the other hand, $x\in\overline{\mathbb E}_E$, i.e.~$R_x(z)\neq0$ for all $z\in\mathbb D^s$, whence
$$
R_{|\alpha|_{\lambda}.x}(z)=R_x(\lambda z)\neq0,\quad z\in\overline{\mathbb D}^s,
$$
i.e.~$|\alpha|_{\lambda}.x\in\mathbb E_E$---a contradiction.

To see $\mathbb E_E$ is hyperconvex, observe that $\log\mathfrak h_{\mathbb E_E}$ is continuous negative plurisubharmonic exhaustion function on $\mathbb E_E$ (cf.~\cite{n2006}, Proposition 1).
\end{proof}

Recall that a set $A\subset\mathbb C^n$ is \emph{polynomially convex} or a \emph{Runge domain} if for every compact $K\subset A$ the polynomial hull $\widehat K$ of $K$ is contained in $A$. Using the same argument as in the Proposition~\ref{prop:lc} we are able to show

\begin{prop}\label{prop:pc}If $r_2=\dots=r_s=1$ then $\mathbb E_E$ is polynomially convex.
\end{prop}

\begin{proof}Without loss of generality we may assume that $s\geq2$. First, using induction on $k$ as before, we prove that $\overline{\mathbb E}_{E_k}\subset\mathbb C^{N_k}$, $k=2,\dots,s$, where $E_k$ is given by (\ref{eq:ek}), are polynomially convex.

First we show that $\overline{\mathbb E}_{E_2}$ is polynomially convex. Take $x_0\notin\overline{\mathbb E}_{E_2}$. We are looking for a polynomial $f$ such that $|f|\leq1$ on $\overline{\mathbb E}_{E_2}$ and $|f(x_0)|>1$.

It follows from Theorem~\ref{thm:ret} that if we write $x=(x',x'')\in\overline{\mathbb E}_{E_2}\subset\mathbb C^{r_1}\times\mathbb C^{N_2-r_1}$ then $x'\in\overline{\mathbb G}_{r_1}$.

Assume first that $x_0'\notin\overline{\mathbb G}_{r_1}$. Since $\overline{\mathbb G}_{r_1}$ is polynomially convex (cf.~\cite{ay2004} for the proof that $\overline{\mathbb G}_2$ is polynomially convex; polynomial convexity of closure of the symmetrized polydisc may be proved in the same way), there is a polynomial $f'$ such that $|f'|\leq1$ on $\overline{\mathbb G}_{r_1}$ and $|f'(x_0')|>1$. Then $f(x',x''):=f'(x')$ is the polynomial we are looking for (use Theorem~\ref{thm:ret}).

Now consider the case $x_0'\in\overline{\mathbb G}_{r_1}$. It follows from Proposition~\ref{prop:clo} that there is a $z_1\in\mathbb D$ such that $|\Psi_{z_1}(x_0',x_0'')|>1$ and $|\Psi_{z_1}(x',x'')|\leq1$ on $\overline{\mathbb E}_{E_2}$. Thus it suffices to approximate the rational function $\Psi_{z_1}$ by polynomials. But it is an immediate consequence of the following

\begin{lem}Let $m=(m_1,\dots,m_n)\in\mathbb N^n$, let $D\subset\mathbb C^n$ be $m$-balanced domain, and let $f:\mathbb D\longrightarrow\mathbb C$ be holomorphic. Then
\begin{equation}\label{eq:1.8.1}
f(z)=\sum_{k=0}^{\infty}Q_k(z),\quad z\in D,
\end{equation}
where
$$
Q_k(z):=\sum_{\alpha\in\mathbb Z_+^n:\langle\alpha,m\rangle=k}\frac{1}{\alpha!}D^{\alpha}f(0)z^{\alpha},\quad z\in\mathbb C^n,
$$
(it is understood $Q_k=0$ if there is no $\alpha\in\mathbb Z_+^n$ with $\langle\alpha,m\rangle=k$); observe that $Q_k:\mathbb C^n\longrightarrow\mathbb C$ is an $m$-homogeneous polynomial of degree $k$, i.e.~$Q_k(m_{\lambda}.z)=\lambda^kQ_k(z)$, $z\in\mathbb C^n$, $\lambda\in\mathbb C$. Moreover, for any compact $K\subset D$ there exist $C>0$ and $\vartheta\in(0,1)$ such that
$$
\|Q_k\|_K\leq C\vartheta^k,\quad k\in\mathbb Z_+.
$$
In particular, the series converges locally normally in $D$.
\end{lem}

The above lemma is well known in the case of balanced domains (cf.~Proposition~1.8.4 in \cite{jp2008}). Since the quasibalanced case may be proved in the same way as the balanced case, we omit here its proof.

Now fix $2\leq k<s$ and assume that $\overline{\mathbb E}_{E_k}$ is polynomially convex. In order to show that $\overline{\mathbb E}_{E_{k+1}}$ is polynomially convex, we take $x_0\notin\overline{\mathbb E}_{E_{k+1}}$ and look for a polynomial $f$ such that $|f|\leq1$ on $\overline{\mathbb E}_{E_{k+1}}$ and $|f(x_0)|>1$. Since $\overline{\mathbb E}_{E_k}$ is assumed to be linearly convex we may proceed as in first inductive step replacing $\overline{\mathbb G}_{r_1}$ with $\overline{\mathbb E}_{E_k}$.

Define
$$
\mathbb E_{E}^{(r)}:=\{|\alpha|_r.x:x\in\overline{\mathbb E}_E\},\quad r\in(0,1).
$$
Observe that $\mathbb E_E^{(r)}$ is polynomially convex and $$
\bigcup_{r\in(0,1)}\mathbb E_{E}^{(r)}=\mathbb E_E.
$$
Now consider any compact $K\subset\mathbb E_E$. Then, for $r$ sufficiently close to 1, $K\subset\mathbb E_{E}^{(r)}$. Since $\mathbb E_E^{(r)}$ is polynomially convex, we have $\widehat K\subset\mathbb E_{E}^{(r)}\subset\mathbb E_E$, i.e.~$\mathbb E_E$ is polynomially convex.
\end{proof}

\section{The $\mu_{1,n}$-quotients}\label{sect:mq}

In 2014 G.~Bharali introduced the following domain
\begin{equation}\label{eq:muquot}
\mathbb E_n:=\left\{(x,y)\in\mathbb C^{n-1}\times\mathbb C^n:\forall_{z,w\in\overline{\mathbb D}}\ Q_x(z)-wP_y(z)\neq0\right\},\quad n\geq2,
\end{equation}
where
$$
P_y(z):=\sum_{j=1}^n(-1)^{j-1}y_jz^{j-1},\quad Q_x(z):=1+\sum_{j=1}^{n-1}(-1)^jx_jz^j.
$$
$\mathbb E_n$ is called \emph{$\mu_{1,n}$-quotient}.

It is a natural generalization of the tetrablock, since $\mathbb E_2=\mathbb E$. On the other hand, $\mu_{1,n}$-quotient is a particular case of the generalized tetrablock. Indeed, $\mathbb E_n=\mathbb E_E$ for $E=\{\diag[z_1\mathbb I_{n-1},z_2]\in\mathbb C^{n\times n}:z_1,z_2\in\mathbb C\}$.

Below we collect the geometric properties of $\mu_{1,n}$-quotients $\mathbb E_n$, $n\geq2$, which are immediate consequence of the results from the previous section.

\begin{thm}\label{thm:ee2en}
\begin{enumerate}
\item[(a)]$\mathbb E_n$ is bounded $(1,2,\dots,n-1,k+1,k+2,\dots,k+n)$-balanced domain, $k\geq0$, but not circled.
\item[(b)]$\mathbb E_n$ is a Hartogs domain in $\mathbb C^{2n-1}$ over $\mathbb G_{n-1}$ with $n$-dimensional balanced fibers.
\item[(c)]$\mathbb E_n$ cannot be exhausted by domains biholomorphic to convex ones.
\item[(d)]Let $n\geq4$. Then $l_{\mathbb E_n}$ is not a distance. In particular, $c_{\mathbb E_n}\not\equiv l_{\mathbb E_n}$.
\item[(e)]$c_{\mathbb E_4}(0,\cdot)\not\equiv k_{\mathbb E_4}(0,\cdot)$. In particular, $c_{\mathbb E_4}(0,\cdot)\not\equiv l_{\mathbb E_4}(0,\cdot)$.
\item[(f)]If $n\geq4$ then $\mathbb E_n$ is neither $\mathbb C$-convex nor starlike about the origin.
\item[(g)]$\mathbb E_n$ is linearly convex and hyperconvex.
\item[(h)]$\mathbb E_n$ is polynomially convex.
\item[(i)]$\mathbb E_3$ is not an analytic retract of the open unit ball of a $\mathcal J^*$-algebra of finite rank.
\end{enumerate}
\end{thm}

\begin{rem}In view of the above results, among the domains $\mathbb E_n$, $n\geq2$, the only interesting examples---from the point of view of the Lempert theorem---are $\mathbb E_2$ and, possibly, $\mathbb E_3$. Recall that $c_{\mathbb E_2}=l_{\mathbb E_2}$ (cf.~\cite{ekz2013}) and $\mathbb E_2$ is $\mathbb C$-convex (cf.~\cite{z2013}). It is an open question whether $\mathbb E_3$ has also these properties.
\end{rem}

\section{The pentablock}\label{sect:p}

Recently J.~Agler, Z.~A.~Lykova, and N.~J.~Young introduced a new domain related do $\mu$-synthesis, called pentablock. Recall that the \emph{pentablock} may be defined as follows (cf.~\cite{aly2014})
$$
\mathcal P:=\left\{(a,s,p)\in\mathbb C^3:(s,p)\in\mathbb G_2:|a|<\left|1-\frac{\frac12s\overline{\beta}}{1+\sqrt{1-|\beta|^2}}\right|\right\},
$$
where
$$
\beta=\beta(s,p):=\frac{s-\overline{s}p}{1-|p|^2}.
$$

Note that $\mathcal P$ is a Hartogs domain over $\mathbb G_2$ with balanced fibers. Moreover, $\mathcal P$ is bounded, nonconvex, $(k,1,2)$-balanced, $k\geq0$, starlike about the origin and polynomially convex (cf.~\cite{aly2014}). Recently \L.~Kosi\'nski in \cite{k2014} showed that $\mathcal P$ is linearly convex. In particular, $\mathcal P$ is pseudoconvex.

Proposition~\ref{prop:x} implies immediately

\begin{thm}\label{thm:3}$\mathcal P$ cannot be exhausted by domains biholomorphic to convex ones.
\end{thm}

\begin{proof}Indeed, since $\mathcal P$ is a Hartogs domain over $\mathbb G_2$, it suffices to take
\begin{equation}\label{eq:penta}
\mathbb G_2\ni(s,p)\overset{\theta}\longmapsto(0,s,p)\in\mathcal P,\quad \mathcal P\ni(a,s,p)\overset{\iota}\longmapsto(s,p)\in\mathbb G_2,
\end{equation}
and observe that $\mathcal S(\mathbb G_2)=\varnothing$ (cf.~\cite{e2013}, Corollary 3).
\end{proof}

Moreover, we have the following simple

\begin{prop}\label{prop:ph}Let $m=(k,1,2)$, $k\geq1$. Then the $m$-Minkowski functional $\mathfrak h_{\mathcal P}$ is continuous. In particular, $\mathcal P$ is hyperconvex.
\end{prop}

\begin{proof}To prove the continuity of $\mathfrak h_{\mathcal P}$ it suffices to show that
$$
\partial\mathcal P\subset\{z\in\mathbb C^3:\mathfrak h_{\mathcal P}(z)=1\}.
$$
Suppose $\mathfrak h_{\mathcal P}(z)>1$ for some $z\in\partial\mathcal P$ and take $0<r<1$ with $\mathfrak h_{\mathcal P}(m_r.z)=r\mathfrak h_{\mathcal P}(z)>1$. In particular, $m_r.z\notin\mathcal P$. On the other hand, $z=(a,s,p)\in\overline{\mathcal P}$, i.e. $(s,p)\in\overline{\mathbb G}_2$ and
$$
|a|\leq\left|1-\frac{\frac12s\overline{\beta}}{1+\sqrt{1-|\beta|^2}}\right|
$$
(cf.~\cite{aly2014}, Theorem 5.3). Write $m_r.z=(r^ka,ra,r^2p)$. Then $(rs,r^2p)\in\mathbb G_2$ and
$$
|r^ka|<|a|\leq\left|1-\frac{\frac12rs\overline{\beta(rs,r^2p)}}{1+\sqrt{1-|\beta(rs,r^2p)|^2}}\right|
$$
(last inequality follows from the fact that $\mathcal P$ is $(0,1,2)$-balanced, i.e.~$(a,rs,r^2p)\in\overline{\mathcal P}$), i.e.~$m_r.z\in\mathcal P$---a contradiction.

To see $\mathcal P$ is hyperconvex, observe that $\log\mathfrak h_{\mathcal P}$ is continuous negative plurisubharmonic exhaustion function on $\mathcal P$ (cf.~\cite{n2006}, Proposition 1).
\end{proof}

\begin{rem}We end this section with some natural questions. Do the Carath\'eodory distance and the Lempert function coincide on the pentablock? Is pentablock a $\mathbb C$-convex domain? Can $\mathcal P$ be exhausted by strongly linearly convex domains?
\end{rem}

\textbf{Acknowledgements.} The author is greatly indebted to \L.~Kosi\'nski for many stimulating conversations.

\end{document}